\documentclass[a4paper]{article}
\usepackage{a4}
\usepackage[T1]{fontenc}
\usepackage{tikz-cd}
\usepackage{amssymb,amsmath,latexsym,amsthm,paralist,setspace,a4,amssymb,mathrsfs,dsfont,MnSymbol,dsfont}
\usepackage{float}
\usetikzlibrary{babel}
\usepackage{varioref}
\usepackage[bookmarksnumbered=true]{hyperref}
\usepackage[capitalize]{cleveref}
\let\noi\noindent

\newcommand{\ms}{\medskip}

\newcommand{\red}{\mathrm{red}}

\newcommand{\Gal}{\mathrm{Gal}}

\newcommand{\Hom}{\mathrm{Hom}}

\newcommand{\coker}{\mathrm{coker}}

\newcommand{\cd}{\mathop{cd}\nolimits}

\newcommand{\Z}{\mathds Z}

\newcommand{\lang}{\longrightarrow}

\renewcommand{\O}{\mathcal{O}}

\newcommand{\Cat}{\mathrm{Cat }}

\newcommand{\et}{\mathrm{et}}
\newcommand{\res}{\mathrm{res}}
\newcommand{\Nis}{\mathrm{Nis}}

\renewcommand{\P}{{\mathbb P}}

\newcommand{\Sh}{\mathit{Sh}}

\newcommand{\cC}{{\mathscr C}}

\newtheoremstyle{alexthm}
  {}
  {}
  {\sl }
  {}
  {\bf}
  {.}
  {.5em}
  {}
\theoremstyle{alexthm}

\newtheorem{theorem}{Theorem}[section]
\newtheorem*{theorem*}{Theorem}
\newtheorem{corollary}[theorem]{Corollary}
\newtheorem{proposition}[theorem]{Proposition}

\newtheorem*{lemma*}{Lemma}

\newtheoremstyle{alexdef}
  {}
  {}
  {\rm }
  {}
  {\bf}
  {.}
  {.5em}
  {}
\theoremstyle{alexdef}
\newtheorem*{example*}{Example}
\newtheorem{example}[theorem]{Example}

\newtheorem{remark}[theorem]{Remark}

\newtheorem{definition}[theorem]{Definition}

\DeclareMathOperator{\Spec}{\textit{Spec}}

\DeclareMathOperator*{\colim}{colim}
\newcommand{\liso}{\mathrel{\hbox{$\longrightarrow$} \kern-2.4ex\lower-1ex\hbox{$\scriptstyle\sim$}\kern1.7ex}}
\title{\bf On the \'{e}tale site of marked schemes}
\author{Alexander Schmidt}
\date{}
\begin{document}
\maketitle

When considering the \'{e}tale site of a scheme it is often of interest to consider a variant which forces a given set of points to split in at least one member of a covering. Examples are the \'{e}tale site of a marked curve used in \cite{Sch}, where a finite set of closed points is considered and the Nisnevich site \cite{Nis}, where all points are required to split. In this note we develop this approach in greater generality. Furthermore, we close a small gap in the literature by showing that any Nisnevich covering of a quasi-compact scheme has a finite subcovering.

\section{Definition of the marked site}\label{defsect}

Let $X$ be a scheme and let $T$ be a set of points of $X$. We will loosely write $T\subset X$ and call the pair $(X,T)$ a \emph{marked scheme}. A morphism $f: (Y,S)\to (X,T)$ of marked schemes is a scheme morphism $f: Y\to X$ with $f(S)\subset T$.

\begin{definition}\label{def:marked site}
Let $(X,T)$ be a marked scheme. The \emph{marked \'{e}tale site} $(X,T)_\et$ consists of the following data:
The category $\Cat (X,T)_\et$ is the category of morphisms $f: (U,S)\to (X,T)$ such that \smallskip
\begin{compactitem}
  \item[\rm a)] $(f: U \to X)$ is \'{e}tale, and
  \item[\rm b)] $S=p^{-1}(T)$.
\end{compactitem}

\smallskip\noindent
A family $(p_i: (U_i,S_i) \to (U,S))_{i\in I}$ of morphisms in $\Cat (X,T)_\et$ is a covering if it is surjective and any point $s\in S$ \emph{splits}, i.e., there exists an index~$i$ and a point $u_i\in S_i$ mapping to $s$ such that the induced field homomorphism $k(s)\to k(u_i)$ is an isomorphism.
\end{definition}

\begin{example}
For $T=\varnothing$, we obtain the small \'{e}tale site of $X$, for $T=X$ the Nisnevich site \cite{Nis}.
\end{example}

A morphism of marked schemes induces a morphism of the associated marked \'{e}tale sites in the obvious way.

\medskip\noindent
We consider the following ``geometric points'' of $(X,T)_\et$: we fix a separable closure $k(x)^s$ of $k(x)$ for every scheme-theoretic point $x\in X$ and consider the following morphisms of marked schemes
\begin{compactitem}
\item[1.)] for  $x\notin T$, the natural morphism $(\Spec k(x)^s,\varnothing)\to (X,T)$.
\item[2.)] for  $x\in T$, the natural morphisms $(\Spec \kappa,\Spec \kappa)\to (X,T)$ for every subextension $\kappa/k(x)$ of $k(x)^s/k(x)$.
\end{compactitem}
If $f:P\to X$ is any of the morphisms described in 1.) and 2.), the assignment $F\mapsto \Gamma(P,f^*F)$ is a topos-theoretic point of $(X,T)_\et$ and one easily verifies that this family of points is conservative. In particular, exactness of sequences of abelian sheaves can be checked stalkwise.

We denote the cohomology of a sheaf $F\in \Sh_\et(X,T)$ of abelian groups on $(X,T)_\et$ by $H^*_\et(X,T,F)$.

\section{Excision}

Let $(X,T)$ be a marked scheme, $i:Z\hookrightarrow X$ a closed immersion and $U=X\smallsetminus Z$ the open complement.
The right derivatives of the left exact functor ``sections with support in $Z$''
\[
F \mapsto \ker (F(X,T) \to F(U,T\cap U))
\]
are called the \emph{cohomology groups with support in $Z$}. Notation:  $H^*_Z(X,T, F)$.

The proof of the next proposition is standard (cf.\ \cite[III,\,(2.11)]{Art} for the \'{e}tale case without marking).

\begin{proposition}\label{longex}
There is a long exact sequence
\begin{multline*}
0 \rightarrow H^0_Z(X,T, F) \rightarrow H^0_\et (X,T, F) \rightarrow H^0_\et(U, T\cap U,F) \rightarrow \\H^1_Z(X,T, F) \rightarrow H^1_\et(X, T, F) \rightarrow H^1_\et(U,T\cap U, F) \rightarrow \ldots \qquad
\end{multline*}
\end{proposition}

\begin{proposition}[Excision] Let $\pi: (X',T') \rightarrow (X,T)$ be a morphism of marked schemes,
$Z \hookrightarrow X$, $Z' \hookrightarrow X'$ closed immersions and $U=X\smallsetminus Z$, $U'=X'\smallsetminus Z'$ the open complements. Assume that

\ms
\begin{compactitem}
\item $\pi: X' \to X$ is \'{e}tale,
\item $T'=\pi^{-1}(T)$,
\item $\pi$ induces an isomorphism $Z'_\red \liso Z_\red $,
\item $\pi(U') \subset U$.
\end{compactitem}

\ms\noi
Then the induced homomorphism
\[
H^p_Z(X,T, F) \liso H^p_{Z'}(X', T', \pi^*F)
\]
is an isomorphism for every sheaf $F \in \Sh_\et(X,T)$ and all $p \geq 0$.
\end{proposition}
\begin{proof} The standard proof for \'{e}tale topology applies: By the general theory, $\pi^*$ is exact.
Since $\pi$ belongs to $\Cat (X,T)_\et$, $\pi^*$ has the exact left adjoint ``extension by zero'',  hence $\pi^*$ sends injectives to injectives. Therefore it suffices to deal with the case $p =0$. Without changing the statement, we can replace all occurring schemes by their reductions.
By assumption,
\[
(X',T') \ \coprod \ (U,T\cap U) \longrightarrow (X,T)
\]
is a covering.
For $\alpha \in H^0_Z (X,T, F)$ mapping to zero in  $H^0_{Z'}(X', T', \pi^*F)$ we therefore obtain $\alpha=0$.

Now let $\alpha' \in H^0_{Z'}(X', T', \pi^*F)$ be given.
We show that $\alpha'$ and $0\in H^0(U,T\cap U, F)$ glue to an element in $H^0_Z(X,T, F)$. The only nontrivial compatibility on intersections is  $p^*_1(\alpha')=p^*_2(\alpha')$ for $p_1,p_2:  (X'\times_XX', T'\times_T T') \to (X',T')$.
This can be checked on stalks noting that $Z'\liso Z$ implies that the two projections $Z'\times_Z Z'\to Z'$ are the same.
\end{proof}

\section{Continuity}
\begin{proposition}\label{noetherian}
Let $X$ be a quasi-compact scheme and let $T\subset X$ be a closed subscheme. Then every \'{e}tale covering of $(X,T)$ admits a finite subcovering.
\end{proposition}

\begin{proof}
Since $X$ has a finite affine Zariski-open covering, we may assume that $X$ is affine, in particular $X$ is quasi-separated. Then also $T$ is quasi-compact and quasi-separated. Let
\[
\coprod_{i\in I} (U_i,S_i) \to (X,T)
\]
be an \'{e}tale covering. Let, for $i\in I$, $K_i\subset T$ be the set of points in $T$ which split in $U_i \to X$. By \cite[Lemma 13.3]{Scr}, $K_i$ is ind-constructible, i.e., open in the constructible topology of $T$, which is compact by \cite[1.9.15\,(iii)]{ega}. Since $T=\bigcup_i K_i$ by assumption, we find a finite subset $J\subset I$ with $T=\bigcup_{i\in J}K_i$.  Furthermore, since $X$ is quasi-compact and \'{e}tale morphisms are open, we find a finite subset $J'\subset I$ such that $\coprod_{i\in J'} U_i \to X$ is an \'{e}tale covering. We conclude that $\coprod_{i\in J\cup J'} (U_i,S_i) \to (X,T)$ is a finite subcovering of $\coprod_{i\in I} (U_i,S_i) \to (X,T)$.
\end{proof}

As in \cite[VII,\, 3.2]{sga4}  for the unmarked \'{e}tale site, we define the \emph{restricted marked \'{e}tale site}
\[
(X,T)_\et^\res
\]
as the restriction of $(X,T)_\et$ to the subcategory of all $(U,S)\in (X,T)_\et$ where $U\to X$ is of finite presentation. Assume that $X$ is quasi-compact and quasi-separated. Then the same is true for any such $U$ and \cref{noetherian} shows that the restricted site is noetherian. Moreover, the categories of sheaves on $(X,T)_\et$ and $(X,T)_\et^\res$ are naturally equivalent. Hence the same argument as in the unmarked \'{e}tale case \cite[VII,\,Prop.\,3.3]{sga4} shows

\begin{theorem}
\label{colimitsheaves} Let $X$ be a quasi-compact and quasi-separated scheme and let $T\subset X$ be a closed subscheme. Let $(F_i)$ be a filtered direct system of abelian sheaves on $(X,T)_\et$. Then
\[
\colim_i \, H^p_\et(X,T, F_i)\cong H^p_\et(X,T, \colim_i \, F_i)
\]
for all $p\ge 0$.
\end{theorem}

Next we consider inverse limits of marked schemes.

\begin{theorem} \label{cont} Let $(X,T)$ be a marked scheme with $T$ closed in $X$ and let $X_i\to X$, $i\in I$, be an inverse system of $X$-schemes.
Assume that all $X_i$ are quasi-separated and quasi-compact and that all transition morphisms are affine. Let $T_i$ be the preimage of\ $T$ in $X_i$ and put $X_\infty=\varprojlim X_i$, $T_\infty= \varprojlim T_i$.

\smallskip
Then the restricted site $(X_\infty,T_\infty)_\et^\res$ is the limit site of the sites $(X_i,T_i)_\et^\res$.
\end{theorem}

\begin{corollary} \label{cohlimitscheme} With the notation and assumptions of \cref{cont}, let $F$ be a sheaf of abelian groups on $(X,T)_\et$. We denote its inverse image on $(X_i,T_i)_\et$ and $(X_\infty, T_\infty)_\et$ by $F_i$ and $F_\infty$.  Then the natural map
\[
\colim_i H^p_\et(X_i,T_i,F_i) \lang H^p_\et(X_\infty,T_\infty,F_\infty)
\]
is an isomorphism for all $p\ge 0$.
\end{corollary}

\begin{proof}[Proof of \cref{cont}]  By \cite[III, Theorem\,3.8]{Art}, the site $(X_\infty)_\et^\res$ is naturally equivalent to the limit site of the $(X_i)_\et^\res$. In view of \cref{noetherian}, it therefore suffices to show that for every quasi-compact \'{e}tale surjection $U_i\to X_i$ with the property that every point of $T_\infty$ splits in $U_\infty= U_i\times_{X_i} X_\infty\to X_\infty$ there exist $j \ge i$  such that every point of $T_j$ splits in $U_j= U_i \times_{X_i} X_j \to X_j$.
We follow the proof of \cite[Lemma 13.2]{Scr} for Nisnevich coverings. By \cite[Lemma 13.3]{Scr}, the subset $S_j \subset T_j$ of points that split in $U_j \to X_j$ is ind-constructible for all $j\ge i$. Denoting the projection by $u_j: T_\infty \to T_j$, the assumption on $U_\infty \to X_\infty$ implies  $T_\infty= \bigcup_j u_j^{-1}(S_j)$. Considering the $T_j\subset X_j$ as reduced, closed subschemes, we may apply \cite[Cor.\,8.3.4]{ega} to obtain $S_j=T_j$ for some $j$.
\end{proof}

\begin{remark} Let $A$ be a ring and let  $(A\to B_i)_{i\in I}$ be an affine Nisnevich covering. We write $A$ as the union of its finitely generated subrings. Then, by \cref{noetherian} and \cref{cont}, there exists a finite subset $J\subset I$, a subring $A'\subset A$ which is finitely generated over $\Z$ and a finite Nisnevich covering $(A' \to B'_{j})_{j\in J} $ such that $B_j\cong A\otimes_{A'} B_j'$ for all $j\in J$.

Hence the refined definition of Nisnevich coverings for general rings introduced by Lurie in \cite[XI, Definition 1.1 and Remark 1.15]{DAG} coincides with the naive definition.
\end{remark}

\begin{corollary} Let $(X,T)$ be a marked scheme with $T$ closed in $X$ and $Z=\{z_1,\ldots,z_n\}$ a finite set of closed points of $X$. Put $X_{z_i}^h=\Spec(\O_{X,z_i}^h)$. Then, for every sheaf $F$ of abelian groups on $(X,T)_\et$ and all $p\ge 0$
\[
H^p_Z(X,T,F) \cong \bigoplus_{i=1}^n H^p_{\{z_i\}}(X_{z_i}^h, T\cap X_{z_i}^h, F).
\]
\end{corollary}

\begin{proof}
Since
$
H^p_Z(X,T,F) \cong \bigoplus_{i=1}^n H^p_{\{z_i\}}(X, T, F)
$,
we may assume that $Z=\{z\}$ consists of a single closed point. Excision shows that
\[
H^p_{\{z\}}(X,T,F)= H^p_{\{z\}}(U,T\cap U,F)
\]
for every affine \'{e}tale open neighbourhood $U$ of $z$. Since $X_z^h$ is the limit over all these $U$, the long exact sequences of \cref{longex} together with \cref{cohlimitscheme} show the result.
\end{proof}

Using \cref{cohlimitscheme}, it is easy to calculate the stalks of the higher direct images of the site morphism $(X,T)_\et\to (X,X)_\et=X_\Nis$. The Leray spectral sequence together with the fact that the Nisnevich cohomological dimension of noetherian schemes is bounded by the Krull dimension \cite[Theorem 1.32]{Nis} yields:

\begin{corollary} \label{cohdim}
Let $X$ be a noetherian scheme of finite Krull dimension $d$, $T\subset X$ closed and assume that there exists a nonnegative integer
$N$ such that
\[
\cd ( k(x)) \leq N
\]
for all points $x\in X\smallsetminus T$. Then for every abelian torsion sheaf  $F$ on $(X,T)_\et$ we have
\[
H^p_\et(X,T,F)=0\quad \hbox{ for } p > N+d.
\]
\end{corollary}

\section{Galois covers}
\begin{definition}
A \emph{Galois cover} of $X$ with finite Galois group $G$ in the site $(X,T)_\et$ is a morphism $(Y,S)\to (X,T)$ in $(X,T)_\et$ together with a right action of $G$ on $Y$ over $X$ such that the following holds:
\begin{enumerate}
  \item $(Y,S)\to (X,T)$ is a covering for the site $(X,T)_\et$.
  \item $Y\to X$ is an \'{e}tale Galois cover, i.e.,
  \[
 Y\times G \to  Y\times_X Y,\quad (y,g) \mapsto (y,yg)
  \]
 is an isomorphism.
\end{enumerate}
\end{definition}

\noindent
Since $G$ acts transitively on the set of points in $Y$ over a given point $x\in X$, we see that every $t\in T$ splits \emph{completely} in $Y/X$.

\begin{proposition}[Hochschild-Serre spectral sequence] Let $(Y,S) \to (X,T)$ be a Galois cover with finite group $G$ und $F\in \Sh_\et(X,T)$. Then there is a natural spectral sequence
\[
E_2^{pq}=H^p(G, H^q_\et(Y,S, F)) \Longrightarrow H^{p+q}_\et(X,T,F).
\]
\end{proposition}

\begin{proof}
The proof is word-by-word the same as for the \'{e}tale cohomology, see \cite[Theorem 2.20]{Mi}.
\end{proof}

\begin{remark} Assume that $X$ is quasi-compact and quasi-separated and $T\subset X$ closed. Let $$ (Y_i,S_i) \to (X,T)$$ be a directed inverse system of Galois covers with finite Galois groups $G_i$ and $(Y,S)=\lim (Y_i,S_i)$. Then $(Y,S)\to (X,T)$ is a pro-Galois cover with profinite Galois group $G=\lim G_i$. By \cref{cont}, for $F\in \Sh_\et(X,T)$, the groups $H^q_\et(Y,S,F)=\colim H^q_\et(Y_i,S_i, F)$ are discrete $G$-modules and we obtain the profinite Hochschild-Serre sequence
\[
E_2^{pq}=H^p(G, H^q_\et(Y,S, F)) \Longrightarrow H^{p+q}_\et(X,T,F),
\]
where $H^*(G,-)$ is the continuous cohomology of the profinite group $G$ with values in a discrete $G$-module (see \cite[I,\,\S2]{NSW}).
\end{remark}

\section{Fundamental group} We recall some facts from Artin-Mazur \cite{AM}. Let $\cC$ be a pointed site and  $\mathrm{HR}(\cC)$ the category of pointed hypercovers of $\cC$ \cite[\S8]{AM}. If $\cC$ is locally connected, then  the ``connected component functor'' $\pi$ defines an object
\[
\Pi \cC = \{\pi(K_\bullet)\}_{K_\bullet\in \mathrm{HR}(\cC)}
\]
in the pro-category of the homotopy category of pointed simplicial sets. By definition, the fundamental group of $\cC$ is the pro-group $\pi_1(\Pi(\cC))$.

Let $X$ be a locally noetherian scheme. Then (cf.\ \cite[\S9]{AM}) the site $X_\et$, and hence also $(X,T)_\et$ is locally connected. Pointing $(X,T)_\et$ by choosing any ``geometric'' point $\bar x$  described at the end of \cref{defsect}, we obtain the \'{e}tale fundamental group $\pi_1^\et(X,T,\bar x)$. It is independent of the choice of $\bar x$ up to isomorphism, which is canonical up to inner automorphisms. By \cite[Cor.\,10.7]{AM}, for any group $G$, the set $\Hom(\pi_1^\et(X,T,\bar x),G)$ is in bijection with the set of isomorphism classes of pointed (over $\bar x$) $G$-torsors in $(X,T)_\et$. In particular, $\pi_1^\et(X,\varnothing,\bar x)$ is the enlarged \'{e}tale fundamental group of \cite[X, \S6]{sga3} and its profinite completion is the usual \'{e}tale fundamental group of $X$ defined in \cite{sga1}.
If $\bar x$ is a geometric point of $X$, then $\pi_1^\et(X,T,\bar x)$ is a factor group of $\pi^{\et}_1(X,\varnothing,\bar x)$, which is profinite for normal $X$ by  \cite[Thm.\,11.1]{AM}. Hence we obtain the following result.

\begin{proposition}\label{profinite}
Let $X$ be a noetherian, normal, connected scheme and $T\subset X$. Then  (for any choice of base point)
$\pi_1^\et(X,T)$ is a profinite group. Its finite quotients are in bijection with the isomorphism classes of finite connected pointed \'{e}tale Galois covers of~$X$ in which every point $t\in T$ splits completely.
\end{proposition}

\begin{example}
For general $(X,T)$, the fundamental group need not be profinite. For example, let $k$ be a field and $N=\P^1_k/(0\sim 1)$ the node over $k$. Then
\[
\pi_1^\et(N,T)\cong\left\{ \begin{array}{cl}
\Z \times \Gal_k, & T=\varnothing\\
\Z, & T=X.
\end{array}\right.
\]
\end{example}
We will use the notation $\hat{\pi}_1^\et(X,T)$ for the profinite completion of $\pi_1^\et(X,T)$, hence we have a completion map $\pi_1^\et(X,T)\to \hat{\pi}_1^\et(X,T)$ which is an isomorphism by \cref{profinite} if $X$ is a noetherian, normal and connected scheme.

\medskip
We end this section with the following observation concerning products.

\begin{proposition}
Let $k$ be a field, $X$ and $Y$ geometrically connected schemes of finite type over $k$ and $S \subset X(k)$, $T\subset Y(k)$ nonempty sets of $k$-rational points. Let $a$ and $b$ be geometric points of $X\smallsetminus S$ and $Y \smallsetminus T$ with values in a common separably closed extension field of $k$. Assume that at least one of the schemes  $X$ and $Y$ is proper over $k$. Then the natural map
\[
\hat \pi_1^\et (X\times_k Y, S\times T, (a,b)) \lang \hat \pi_1^\et (X, S,a) \times \hat \pi_1^\et (Y, T,b)
\]
is an isomorphism of profinite groups.
\end{proposition}

\begin{proof} We omit the base points from notation. For a connected scheme $X$, let $\widetilde{X}$ denote the profinite universal cover. For a subset $S\subset X$, the kernel of $\hat\pi_1^\et(X) \to \hat\pi_1^\et(X,S)$ is the (closed) normal subgroup of $\hat\pi_1^\et(X)=\Gal(\widetilde{X}|X)$ generated by the decomposition groups of the points in $S$, i.e., it is the (closed) subgroup of $\Gal(\widetilde{X}|X)$ generated by all automorphisms which fix a point $\tilde{s}\in \widetilde{X}$ lying over some $s\in S$. We denote this group by $K(X,S)$.

Now assume we are in the situation of the proposition. By the topological invariance of the \'{e}tale topology we may assume that $k$ is perfect.
Let $\bar k$ be an algebraic closure of~$k$. We denote the base changes to $\bar k$ by $(\bar X, \bar S)$ and $(\bar Y, \bar T)$. By  \cite[X,\ 1.7]{sga1}, we have a natural isomorphism
\[
\hat \pi_1^\et(\bar X \times_{\bar k} \bar Y) \liso  \hat \pi_1^\et(\bar X) \times \hat \pi_1^\et(\bar Y).
\]
Moreover, by \cite[IX,\, 6.1]{sga1}, we have a natural exact sequence
 \[
1 \lang \hat \pi_1^\et(\bar X) \lang \hat \pi_1^\et(X) \lang \Gal(\bar k|k)\lang 1.
 \]
This and the similar sequence for $Y$  shows the isomorphism
\[
\hat \pi_1^\et (X\times_k Y) \liso \hat \pi_1^\et(X) \times_{\Gal(\bar k|k)} \hat{\pi}_1^\et (Y),\eqno (*)
\]
where the term on the right hand side is a fibre product in the category of profinite groups. We consider the corresponding diagram of \'{e}tale Galois covers.
\[
\begin{tikzcd}
&{\widetilde{{X\times_k Y}}}\arrow[d,"\wr"]\\
                                                                                       & \widetilde{X} \times_{\bar{k}} \widetilde{Y} \arrow[ld, "\pi_1^\et(\bar Y)"'] \arrow[rd, "\pi_1^\et(\bar X)"] &                                                                                        \\
\widetilde X \times_{\bar k}\bar Y \arrow[rd] \arrow[rdd, "\pi_1^\et(X)"', bend right] &                                                                                                                                                   & \bar X \times_{\bar k} \widetilde{Y} \arrow[ld] \arrow[ldd, "\pi_1^\et(Y)", bend left] \\
                                                                                       & \bar X \times_{\bar k} \bar Y \arrow[d, "\Gal(\bar k|k)"]                                                                                            &                                                                                        \\
                                                                                       & X \times_k Y                                                                                                                                      &
\end{tikzcd}
\]

 Let $(s,t)\in S\times T \subset X\times_kY$ and let $(\tilde{s},\tilde{t})\in \widetilde{{X\times_k Y}}=\widetilde{X}\times_{\bar k}\widetilde{Y}$ be a point lying above $(s,t)$.
An element $\sigma \in \hat\pi_1^\et(X\times_k Y)= \Gal(\widetilde{{X\times_k Y}}|X\times_kY)$ fixes $(\tilde s,\tilde t)$ if and only if its image in $\hat\pi_1^\et(X)=\Gal(\tilde{X}\times_{\bar k} \bar Y|X\times_kY)$ fixes $\tilde{s}\in \widetilde{X}$ and its image in  $\hat\pi_1^\et(Y) = \Gal(\bar{X}\times_{\bar k} \tilde Y|X\times_kY)$ fixes $\tilde t\in \tilde Y$. Hence, the isomorphism $(*)$ induces an isomorphism of subgroups
 \[
 K(X\times_k Y, S\times T) \liso K(X,S) \times_{\Gal(\bar k|k)} K(Y,T). \eqno (**)
 \]
The isomorphisms $(*)$ and $(**)$ together induce an isomorphism
\[
\hat \pi_1^\et(X\times_k Y, S\times T) \liso C
\] with
 \[
 C=\coker \big(K(X,S) \times_{\Gal(\bar k|k)} K(Y,T) \lang \hat \pi_1^\et(X) \times_{\Gal(\bar k|k)} \hat{\pi}_1^\et (Y)\big).
 \]
The natural homomorphism $C\to \hat\pi_1^\et(X,S)\times \hat\pi_1^\et(Y,T)$ is injective. To conclude the proof of the proposition, it remains to show surjectivity, i.e., we have to show that every element in $\hat\pi_1^\et(X,S)\times \hat\pi_1^\et(Y,T)$ has a preimage in $\hat \pi_1^\et(X) \times_{\Gal(\bar k|k)} \hat{\pi}_1^\et (Y)\subset \hat \pi_1^\et(X) \times \hat{\pi}_1^\et (Y)$. For this it suffices to show that the composite map $K(X,S)\hookrightarrow \pi_1^\et(X) \to \Gal(\bar k | k)$ is surjective. This is true since $K(X,S)$ contains the decomposition group of a $k$-rational point.
\end{proof}

\section{A modification}
We consider a modification of the marked \'{e}tale site which was used in \cite{Sch} for one-dimensional, noetherian regular schemes.
\begin{definition}
The \emph{strict marked \'{e}tale site} $(X,T)_{\textrm{et-s}}$ consists of the following data:
$\Cat (X,T)_{\textrm{et-s}}$ is the category of morphisms $f: (U,S)\to (X,T)$ such that \smallskip
\begin{compactitem}
  \item[\rm a)] $(f: U \to X)$ is \'{e}tale,
  \item[\rm b)] $S=p^{-1}(T)$, and
  \item[\rm c)] for every $u\in S$ mapping to $t\in T$ the induced field homomorphism $k(s)\to k(u)$ is an isomorphism.
\end{compactitem}
Coverings are surjective families.
\end{definition}

\begin{proposition} {\rm (i)}
If $T\subset X$ consists of a finite set of closed points, then the natural morphism of sites $\varphi: (X,T)_\et \to (X,T)_{\mathrm{et-s}}$ induces isomorphisms
\[
H^p_{\mathrm{et-s}}(X,T,F) \liso H^p_{\et}(X,T,\varphi^*F),\ H^p_{\mathrm{et-s}}(X,T,\varphi_* G) \liso H^p_{\et}(X,T,G)
\]
for any $F\in \Sh_{\mathrm{et-s}}(X,T)$, $G\in \Sh_{\mathrm{et-s}}(X,T)$ and  $p\ge 0$.

\ms\noi {\rm (ii)}
For locally noetherian $X$ (and any chosen base point), the natural map
\[
\pi_1^\et(X,T) \lang \pi_1^{\mathrm{et-s}}(X,T)
\]
is an isomorphism.
\end{proposition}
\begin{proof}
Let $(U,S)\in \Cat (X,T)_{\textrm{et-s}}$ and assume that $(f_i: (U_i,S_i)\to (U,S))$ is a covering in $(X,T)_\et$. Removing for all $i$ the finitely many points $s\in S_i$ such that $k(f(s_i))\to k(s_i)$ is not an isomorphism from $U_i$, we obtain a strict covering $(f_i: (U_i',S_i')\to (U,S))$ which is a refinement of the original one. Hence $\varphi_*\varphi^*F=F$ and $R^q\varphi_* G=0$ for $q\ge 1$. In view of the Leray spectral sequence, this shows (i). Assertion (ii) follows since both pro-groups represent the same functor: for any group $G$, a $G$-torsor in $(X,T)_\et$ is the same as a $G$-torsor in $(X,T)_{\textrm{et-s}}$.
\end{proof}

\vskip1cm\noindent
\emph{Acknowledgement.} The author thanks Philippe Lebacque for helpful comments and motivational discussions.

\ms\noindent
{\small
Alexander Schmidt, Mathematisches  Institut, Universit\"{a}t Heidelberg, Im Neuenheimer Feld 205, 69120 Heidelberg, Deutschland\\
email: {\tt schmidt@mathi.uni-heidelberg.de}}
\end{document}